\documentclass[final]{siamltex}

\usepackage{graphics}
\usepackage{amssymb, latexsym, amsmath, epsfig}
\usepackage{graphicx}
\usepackage{multirow}
\usepackage{threeparttable}

\def\e{{\epsilon}}

\newtheorem{coro}[theorem]{Corollary}

\newtheorem{remark}[theorem]{Remark}

\def\cal{\mathcal}

\def\e{\varepsilon}

\def\O{\Omega}

\newcommand\ddelta\bigtriangledown
\newcommand\ld\lambda
\newcommand\Ld\Lambda

\def \bN {\Bbb N}
\def \bP {\Bbb P}

\def \bZ {\Bbb Z}

\begin{document}
\title{Any order superconvergence finite volume schemes \\
for 1D general elliptic equations}
\author{Waixiang Cao
    \thanks{College of Mathematics and Scientific Computing, Sun Yat-sen University,
            Guangzhou, 510275, P. R. China.}
\and Zhimin Zhang
\thanks{Department of Mathematics, Wayne State University, Detroit, MI 48202, USA.
     This author is partially supported by the US National Science Foundation through grant
     DMS-111530, the Ministry of Education of China through the Changjiang Scholars program,
    and Guangdong Provincial Government of China through the
   ``Computational Science Innovative Research Team" program.}
\and Qingsong Zou
\thanks{College of Mathematics and Scientific Computing and Guangdong Province Key Laboratory of Computational Science, Sun Yat-sen
          University, Guangzhou, 510275, P. R. China. This author is supported in part by
          the National Natural Science Foundation of China under the grant 11171359 and  in part by the Fundamental Research Funds for the Central
          Universities of China.}
}
\maketitle
\begin{abstract}
We present and analyze a finite volume scheme of arbitrary order
for elliptic equations in the one-dimensional setting.
In this scheme, the control volumes are constructed by using
the Gauss points in subintervals of the underlying mesh.
We provide a unified proof for the inf-sup condition,
and show that our finite volume scheme has optimal convergence rate
under the energy and $L^2$ norms of the approximate error.
Furthermore, we prove that the derivative error is superconvergent
at all Gauss points and in some special case, the convergence rate can reach $h^{2r}$, where $r$ is the polynomial degree
of the trial space. All theoretical results are justified by numerical tests.
\end{abstract}

\section{Introduction}
\setcounter{equation}{0}
The {\it finite volume method} (FVM) attracted a lot of attentions during the past
several decades,  we refer to \cite{Bank.R;Rose.D1987,Barth.T;Ohlberger2004, Cai.Z1991, Cai.Z_Park.M2003, ChenWuXu2011, Ewing.R;Lin.T;Lin.Y2002, EymardGallouetHerbin2000, Li.R2000, Ollivier-Gooch;M.Altena2002,Plexousakis_2004,Suli1991, Xu.J.Zou.Q2009, zou2009a} and the
references cited therein for an incomplete list of references.  Due to the local conservation of numerical fluxes, the capability to deal with the problems on the domains with complex geometries,
 and other advantages, FVM has a wide range of applications
in scientific and engineering computations  (see, e.g., \cite{EymardGallouetHerbin2000,Ollivier-Gooch;M.Altena2002}).

   There have been many studies of the mathematical theory for FVM, see, e.g.,
   \cite{Bank.R;Rose.D1987, Xu.J.Zou.Q2009} and the monographs \cite{Barth.T;Ohlberger2004,EymardGallouetHerbin2000,Li.R2000}.
   However, much attention has been paid to linear FVM schemes(see e.g., \cite{Bank.R;Rose.D1987, Cai.Z1991, Ewing.R;Lin.T;Lin.Y2002,Suli1991,Suli1992}),  high order FVM schemes have not been investigated
   as much or as satisfactorily as  linear FVM schemes. Moreover, the analysis of high order FVM schemes are often done case by case. For instances, earlier works on quadratic FVM schemes  can be traced back to \cite{Tian.Chen_1991,Emonot1992,Liebau1996},  high order FVMs  for  1D elliptic equations were derived in \cite{Plexousakis_2004},
   and high order FVMs on rectangular meshes were derived
   and analyzed in \cite{Cai.Z_Park.M2003},  the quadratic FVM schemes on triangular meshes have also been intensively studied by \cite{Li.R2000,Xu.J.Zou.Q2009,Chen.L}. To the best of our knowledge, no analysis of FVM scheme
   of an arbitrary order has been published yet.

   In this paper, we study a family of any order  FVM schemes in the one-dimensional setting.
   Instead of a case-by-case study as in the literature for quadratic and cubic FVM schemes,
   we adopt a unified approach to establish the inf-sup condition.
   Earlier works(see, e.g. \cite{Liebau1996, Li.R2000,Xu.J.Zou.Q2009,ChenWuXu2011}) utilized element-wise analysis to prove that the bilinear form resulting from FVM is positive definite,
   which is a stronger property than the inf-sup condition. Hence, some assumption is needed
   for the mesh, such as quasi-uniformity and shape-regularity (in 2D). The major difference here
   is that we prove only the inf-sup condition (instead of positive definiteness of the bilinear form)
   and there is no mesh condition attached. With help of the inf-sup condition,
   we obtain the optimal rate of convergence in both the $H^1$ and $L^2$ norms.

   In this paper, we also study the superconvergence property of our FVM schemes.   Note that while the superconvergence theory of {\it finite element methods} (FEM) has reached its maturity
   (\cite{Babuvska.I;Strouboulis.T2001,Chen.C.M2001, L.R.Wahlbin1995,Zhu.QD;LinQ1989,Zienkiewicz;zhujianzhong2005}),
   the superconvergence analysis of  FVM has also been focused on the  linear schemes (see, e.g., \cite{Cai.Z1991, Xu.J.Zou.Q2009}).
   In this paper,  it is shown that for a 1D general elliptic  equation,   the superconvergence behavior of  FVM
   is similar to that of the counterpart finite element method. For instances,  the convergence rate at nodal points is  $h^{2r}$,
   the rate at interior Lobatto points (defined in Section 4) is $h^{r+2}$, the convergence rate of the derivative error at Gauss points  is   $h^{r+1}$.
   While in some special cases, some surprising superconvergence results have been found and proved.   That is, the convergence rate of the derivative error at all Gauss points
   can reach $h^{r+2}$ or $h^{2r}$, depending on the coefficient of the elliptic equations. The derivative convergence rate $h^{2r}$ doubles the global optimal rate $h^r$,
  which is much better than the counterpart finite element method's $h^{r+1}$ rate; the derivative convergence rate $h^{r+2}$ is one order higher than the counterpart finite element method's $h^{r+1}$.
  %


    We organize the rest of the paper as follows.
    In Section 2 we present an  arbitrary order FVM scheme for elliptic equations in one-dimensional setting.
    In particular, we use the Gauss points of order $r\ge 1$ to construct the control volumes
    and choose the trial space as the Lagrange finite element of $r$th order
    with the interpolation points being the Lobatto points of order $r$.
    In Section 3 we provide a unified proof for the inf-sup condition and establish the optimal convergence rate both under
    $H^1$ and $L^2$ norms.  In Section 4, we study the superconvergence property at some special points of our FVM schemes of any order.
    In Section 5,  a post-processing technique based on superconvergence results in the section 4 is proposed to recover the derivative of the solution.
    Numerical experiments supporting our theory are presented in Section 6. Some concluding remarks are provided in Section 7.

   In the rest of this paper, ``$A\lesssim B$" means that $A$ can be
  bounded by $B$ multiplied by a constant which is independent of
  the parameters which $A$ and $B$ may depend on. ``$A\sim B$" means $``A\lesssim B"$ and $``B\lesssim A"$.

\bigskip
\section{FVM schemes of any order}
\setcounter{equation}{0}
In this section, we develop finite volume schemes for the following two-point boundary value problem on the interval $\Omega=(a,b)$:
\begin{eqnarray}\label{Poisson}
\begin{aligned}
     - (\alpha u')'(x) +\beta(x) u'(x)+\gamma(x) u(x)=f(x),&&\ \forall x\in \Omega, \\
      u(a)=u(b)=0,&&
\end{aligned}
\end{eqnarray}
  where $\alpha \ge \alpha_0>0$, $\gamma-\frac{1}{2}\beta'\geq \kappa>0 $, $\alpha, \beta, \gamma \in L^{\infty}(\bar{\Omega})$, $f$ is a real-valued function defined on $\bar{\Omega}$.

We first introduce the primal partition and its corresponding trial space. For a positive integer N, let $\bZ_N:=\{1,\cdots,N\}$ and $a=x_0< x_1<\ldots<x_N=b$ be $N+1$ distinct points on $\bar{\Omega}$. For all $i \in \bZ_N$, we denote $\tau_i=[x_{i-1},x_i]$ and  $h_i=x_i-x_{i-1}$, let  $h=\max\limits_{i\in \bZ_N}h_i$ and
\[
  \cal {P}= \{\tau_i: i\in \bZ_N\}
\]
  be a partition of $\Omega$.
The  corresponding trial space is chosen as the Lagrange finite element of $r$th order, $r\geq 1$, defined by
\[
   U^r_{\cal P}=\{v\in C(\Omega): v|_{\tau_j}\in \mathbb{P}_r, \forall\tau_j\in{\cal P}, v|_{\partial\Omega}=0\},
\]
  where $\mathbb{P}_r$ is the set of all polynomials of degree no more than $r$. Obviously, $\dim U^r_{\cal P}=Nr-1$.

  We next present a dual partition and its corresponding test space. Let $G_1,\ldots,G_r$ be $r$ Gauss points, i.e., zeros of the Legendre polynomial of $r$th degree,
   on the interval $[-1,1]$.
  The Gauss points on each interval $\tau_i$ are defined as the affine transformations of
  $G_j$ to $\tau_i$, that is,
\[
  g_{i,j}=\frac {1}{2}(x_i+x_{i-1}+h_iG_j), \quad j\in\bZ_r.
\]
  With these Gauss points,  we construct a dual partition
\[
\cal P'=\{\tau'_{1,0}, \tau'_{N,r}\}\cup \{\tau'_{i,j}: (i,j)\in \bZ_N \times \bZ_{r_i}\},
\]
  where
\[
  \tau'_{1,0}=[0,g_{1,1}], \tau'_{N,r}=[g_{N,r},1],\tau'_{i,j}=[g_{i,j}, g_{i,j+1}],
\]
here
\[
  r_i=\left\{\begin{array}{lll}
 r & \text{if} &i \in \bZ_{N-1}\\
r-1 & \text{if} &i=N
\end{array}
\right.
\text{and} \quad g_{i,r+1}=g_{i+1,1},\forall i \in\bZ_{N-1}.
\]
The test space $V_{\cal P'}$ consists of the piecewise constant functions with respect
to the partition $\cal P'$, which vanish on the intervals $\tau'_{1,0}\cup \tau'_{N,r}$.
In other words,
\[
V_{\cal P'}=\text{Span}\left\{\psi_{i,j}: (i,j)\in \bZ_N \times \bZ_{r_i}\right\},
\]
where $\psi_{i,j}=\chi_{[g_{i,j}, g_{i,j+1}]} $ is the characteristic function on the interval $\tau'_{i,j}$. We find that $\dim V_{\cal P'}=Nr-1=\dim U^r_{\cal P}$.

We are ready to present our finite volume schemes. Integrating  \eqref{Poisson}
on each control volume $[g_{i,j}, g_{i,j+1}], (i,j)\in \bZ_N \times \bZ_{r_i}$ yields
\[
  \int_{g_{i,j}}^{g_{i,j+1}} - (\alpha u')'(x)+\beta(x)u'(x)+\gamma(x)u(x)dx =\int_{g_{i,j}}^{g_{i,j+1}}f(x){dx}.
\]
 In other words,
\begin{equation}\label{conserve}
  \alpha(g_{i,j})u'(g_{i,j})-\alpha(g_{i,j+1})u'(g_{i,j+1})+\int_{g_{i,j}}^{g_{i,j+1}}\big( \beta(x)u'(x)+\gamma(x)u(x)\big)dx =\int_{g_{i,j}}^{g_{i,j+1}}f(x){dx}.
\end{equation}
 Let $w_{\cal P'}\in V_{\cal P'}$,  $w_{\cal P'}$ can be represented as
\[
w_{\cal P'}=\sum_{i=1}^{N}\sum_{j=1}^{r_i}w_{i,j}\psi_{i,j},
\]
 where $w_{i,j}'s$ are constants. Multiplying \eqref{conserve} with $w_{i,j}$ and then summing up for all $(i,j)\in \bZ_N\times \bZ_{r_i}$, we obtain
\begin{equation*}
\begin{split}
 \sum_{i=1}^{N}\sum_{j=1}^{r_i} w_{i,j}\left(\alpha(g_{i,j})u'(g_{i,j})-\alpha(g_{i,j+1})u'(g_{i,j+1})+\int_{g_{i,j}}^{g_{i,j+1}}\big(\beta(x)u'(x)+\gamma(x)u(x)\big)dx \right)
 \\=\int_a^b f(x)w_{\cal P'}(x) dx,
 \end{split}
\end{equation*}
  or equivalently,
\begin{equation*}
\begin{split}
   \sum_{i=1}^{N}\sum_{j=1}^{r} [w_{i,j}]\alpha(g_{i,j})u'(g_{i,j})+\sum_{i=1}^{N}\sum_{j=1}^{r_i} w_{i,j}\left(\int_{g_{i,j}}^{g_{i,j+1}}\big(\beta(x)u'(x)+\gamma(x)u(x)\big)dx \right)
   \\=\int_a^b f(x)w_{\cal P'}(x) dx,
\end{split}
\end{equation*}
   where $[w_{i,j}]=w_{i,j}-w_{i,j-1}$ is the jump of $w$ at the point $g_{i,j}, (i,j)\in \bZ_N \times \bZ_r$ with $w_{1,0}=0, w_{N,r}=0$ and $w_{i,0}=w_{i-1,r},2\leq i\leq N$.

We define the FVM bilinear form for all $v\in H^1_0(\Omega), w_{\cal P'}\in V_{\cal P'}$ by
\begin{equation}\label{bilinear1}
\begin{split}
   a_{\cal P}(v,w_{\cal P'})=&\sum_{i=1}^{N}\sum_{j=1}^{r} [w_{i,j}]\alpha(g_{i,j})v'(g_{i,j})\\
   &+\sum_{i=1}^{N}\sum_{j=1}^{r_i} w_{i,j}\left(\int_{g_{i,j}}^{g_{i,j+1}}\big(\beta(x)v'(x)+\gamma(x)v(x)\big)dx \right).
\end{split}
\end{equation}

  The finite volume method for solving equation \eqref{Poisson} reads as :  Find $u_{\cal P}\in U^r_{\cal P}$ such that
\begin{equation}\label{bilinear}
   a_{\cal P}(u_{\cal P},w_{\cal P'})=(f,w_{\cal P'}),\ \ \forall w_{\cal P'}\in V_{\cal P '}.
\end{equation}



\bigskip

\section{Convergence Analysis}
\setcounter{equation}{0}
In this section, we prove the inf-sup condition and use it to establish some convergence properties of the FVM solution.

\subsection{Inf-sup condition}
We begin with some notations and definitions. First we introduce the {\it broken} Sobolev space
\[
    W^{m,p}_{\cal P}(\Omega)=\{ v\in C(\Omega) : v|_{\tau_i}\in W^{m,p}(\Omega), \forall \tau_i \in \cal P \},
\]
where   $m$ is a given positive integer and $1\leq p \leq \infty$. When $p=2$, we denote $H^m_{\cal P}$ for simplicity. For all $j\geq0$, we define a semi-norm by
\[
    |v|_{j,p,\cal P} = \left( \sum_{\tau_i\in \cal P} |v|_{j,p,\tau_i}^p  \right)^{1\over p}
\]
and  a norm by
\[
    \|v\|_{m,p,\cal P} =\left( \sum_{j=0}^m |v|_{j,p,\cal P}^p \right)^{1\over p}.
\]
We often use $|\cdot|_{m,\cal P}$  instead of $|\cdot|_{m, 2,\cal P}$ and  $\|\cdot\|_{m,\cal P}$ instead of $\|\cdot\|_{m, 2,\cal P}$ for simplicity.

Secondly, for all $w_{\cal P'}\in V_{\cal P'},\ w_{\cal P'}=\sum\limits_{i=1}^N\sum\limits_{j=1}^{r_i} w_{i,j}\psi_{i,j}$,
let
\[
\big|w_{\cal P'}\big|^2_{1,\cal P'}=\sum_{i=1}^{N}\sum_{j=1}^r h_i^{-1}[w_{i,j}]^2,\quad
 \big\|w_{\cal P'}\big\|^2_{0,\cal P'}= \sum_{i=1}^{N}\sum_{j=1}^{r_i} h_i w_{i,j}^2
\]
and
\[
     \big\|w_{\cal P'}\big\|_{\cal P'}^2 = \big|w_{\cal P'}\big|_{1,\cal P'}^2 + \big\|w_{\cal P'}\big\|_{0,\cal P'}^2.
\]
Noticing that $w_{1,0}=w_{N,r}=0$, it is easy to obtain
the following Poincar\'e type inequality
\begin{eqnarray}\label{Poincare}
    \big\|w_{\cal P'}\big\|_{0,\cal P'} \lesssim  \big|w_{\cal P'}\big|_{1,\cal P'}, \quad \forall w_{\cal P'}\in V_{\cal P'},
\end{eqnarray}
where the hidden constant  depends only on $\Omega$ and $r$.

Thirdly, we denote by $A_j, j\in \bZ_r$ the weights of the Gauss quadrature
\[
Q_r(F)=\sum_{j=1}^r A_j F(G_j)
\]
for computing the integral
\[
I(F)=\int_{-1}^1 F(x) dx.
\]
It is well-known that $Q_r(F)=I(F)$ for all $F\in \bP_{2r-1}(-1,1)$.
Naturally, the weights on interval $\tau_i,  i\in \bZ_N$ are
\[
A_{ij}=\frac{h_i}{2}A_j,\quad j\in \bZ_r.
\]

Before the presentation of the inf-sup condition, we define a linear mapping $\Pi_{\cal P}:U_{\cal P}^r \rightarrow V_{\cal P'}$ by
\[
    \Pi_{\cal P}v_{\cal P}=\sum_{i=1}^N \sum_{j=1}^{r_i} v_{i,j}\psi_{i,j},
\]
where the coefficients $v_{i,j}$ are determined by the constraints
\[
    [v_{i,j}]=A_{i,j}v'_{\cal P}(g_{i,j}),\ \ (i,j)\in \bZ_N\times \bZ_{r_i}.
\]
 Note that $v_{\cal P}\in U^r_{\cal P}$, the derivative $v'_{\cal P}\in \bP_{r-1}(\tau_i), i\in\bZ_N$, then
\[
\sum_{i=1}^N\sum_{j=1}^{r} A_{i,j}v_{\cal P}'(g_{i,j})=\int_a^b v_{\cal P}'(x) dx=v_{\cal P}(b)-v_{\cal P}(a)=0.
\]
Therefore,
\begin{eqnarray*}
v_{N,r-1}&=&\sum_{i=1}^N\sum_{j=1}^{r_i}[v_{i,j}]
=\sum_{i=1}^N\sum_{j=1}^{r} A_{i,j}v'_{\cal P}(g_{i,j})-A_{N,r}v'_{\cal P}(g_{N,r})\\
&=&-A_{N,r}v'_{\cal P}(g_{N,r}).
\end{eqnarray*}
In other words, we also have
\[
[v_{N,r}]=v_{N,r}-v_{N,r-1}=A_{N,r}v'_{\cal P}(g_{N,r}).
\]
Consequently,
\begin{eqnarray*}
|\Pi_{\cal P}v_{\cal P}|_{1,\cal P'}^2&=&\sum_{i=1}^{N}\sum_{j=1}^r h_i^{-1}[v_{i,j}]^2
=\sum_{i=1}^{N}\sum_{j=1}^r h_i^{-1} \left(A_{i,j}v'_{\cal P}(g_{i,j})\right)^2 \thicksim\int_a^b \left(v'_{\cal P}(x)\right)^2dx.
\end{eqnarray*}
Namely, we have
\begin{eqnarray}\label{equ-norm}
|\Pi_{\cal P}v_{\cal P}|_{1,\cal P'}\thicksim |v_{\cal P}|_{1,\cal P}.
\end{eqnarray}

With all these preparations, we are now ready to present the inf-sup condition of $a_{\cal P}(\cdot,\cdot)$.

\begin{theorem}
 Assume that the mesh size $h$ is sufficiently small, then
  \begin{equation}\label{infsup}
\inf_{v_{\cal P}\in U^r_{\cal P}}\sup_{w_{\cal P'} \in V_{\cal P'} }
\frac{a_{\cal P} (v_{\cal P} ,w_{\cal P'} )}{\|v_{\cal P} \|_{1,\cal P} \|w_{\cal P'}\|_{\cal P'}}\ge c_0,
\end{equation}
where $c_0>0$ is a constant depending only on $r,\alpha_0,\kappa$ and $\Omega$.
\end{theorem}
\begin{proof} It follows from the bilinear form \eqref{bilinear1} that
\[
a_{\cal P}(v_{\cal P},\Pi_{\cal P}v_{\cal P})=I_1+I_2,\ \ \forall v_{\cal P}\in U_{\cal P}^r
\]
  with
\[
I_1=\sum_{i=1}^{N}\sum_{j=1}^{r} [v_{i,j}]\alpha(g_{i,j})v'_{\cal P}(g_{i,j}),\quad I_2=\sum_{i=1}^{N}\sum_{j=1}^{r_i} v_{i,j}\int_{g_{i,j}}^{g_{i,j+1}}\big(\beta(x)v'_{\cal P}(x)+\gamma(x)v_{\cal P}(x)\big)dx .
\]
Since $(v'_{\cal P})^2 \in \bP_{2r-2}(\tau_i), i\in \bZ_N$, we have
\[
    \int_{x_{i-1}}^{x_i} (v'_{\cal P}(x))^2 dx = \sum_{j=1}^r A_{i,j}(v'_{\cal P}(g_{i,j}))^2.
\]
Therefore,
\begin{eqnarray*}
I_1\ge  \alpha_0\sum_{i=1}^{N}\sum_{j=1}^{r} A_{i,j}(v'_{\cal P}(g_{i,j}))^2= \alpha_0|v_{\cal P}|_{1,\cal P}^2.
\end{eqnarray*}

We next estimate $I_2$. Let $V(x)=\int_a^x \left( \beta(s)v'_{\cal P}(s)+\gamma(s)v_{\cal P}(s) \right) ds $ and denote by
\begin{eqnarray*}
    E_{i}&=&\int_{x_{i-1}}^{x_{i}} v'_{\cal P}(x)V(x)dx-\sum_{j=1}^{r} A_{i,j} v'_{\cal P}(g_{i,j})V(g_{i,j}),
\end{eqnarray*}
the error of Gauss quadrature in the interval $\tau_i$, $i\in\bZ_N$.
Then
\begin{eqnarray*}
I_2&=&- \sum_{i=1}^{N}\sum_{j=1}^{r}  [v_{i,j}] V(g_{i,j})= - \int_a^b v'_{\cal P}(x)V(x)dx + \sum_{i=1}^N E_i.
\end{eqnarray*}
Using the fact that $v_{\cal P}(a)=v_{\cal P}(b)=0$ and
\[
    \int_a^b \beta(x)v'_{\cal P}(x)v_{\cal P}(x)dx=-\frac{1}{2}\int_a^b \beta'(x)v^2_{\cal P}(x)dx,
\]
we obtain
\begin{equation}\label{est3}
    -\int_a^b v'_{\cal P}(x)V(x)dx =\int_a^b \left( \gamma(x)-\frac{\beta'(x)}{2}\right)v^2_{\cal P}(x)dx \geq \kappa \|v_{\cal P}\|_0^2.
\end{equation}
On the other hand, by\cite{DavisRabinowitz1984}(p98, (2.7.12)), for all $i\in \bZ_N$
\[
E_{i}=\frac{h_i^{2r+1}(r!)^4}{(2r+1)[(2r)!]^3}(v'_{\cal P}V)^{(2r)}(\xi_i),
\]
where $\xi_i\in \tau_i$.
By the  Leibnitz formula of derivatives, we have
\begin{equation*}
    \left|(v'_{\cal P}V)^{(2r)}(\xi_i)\right|=\sum_{k=r+1}^{2r} \binom{2r}{k}\left| (\beta v'_{\cal P}+\gamma v_{\cal P})^{(k-1)} (v'_{\cal P})^{(2r-k)}(\xi_i)\right| \ge c_1 \|v_{\cal P}\|_{r,\infty,\tau_i}^2
\end{equation*}
with
\[
   c_1=\max\left\{\|\beta\|_{2r-1,\infty,\tau_i}^2, \|\gamma\|_{2r-1,\infty,\tau_i}^2\right\}\sum_{k=r+1}^{2r} \binom{2r}{k}.
\]
Therefore, by the inverse inequality that
\[
  \|v_{\cal P}\|_{r,\infty,\tau_i} \lesssim h_i^{-(r-\frac12)}|v_{\cal P}|_{1,\tau_i},
\]
  we have
\begin{equation*}\label{quad err}
|E_{i}| \leq \frac{c_1(r!)^4}{(2r+1)[(2r)!]^3} h_i^2 |v_{\cal P}|_{1,\tau_i}^2.
\end{equation*}
Combining the above estimates, we obtain
\[
    I_2 \geq \kappa \|v_{\cal P}\|^2_{0,\cal P}-\frac{c_1(r!)^4}{(2r+1)[(2r)!]^3}h^2\left|v_{\cal P}\right|^2_{1,\cal P}.
\]

Then for sufficiently small $h$, we have
\begin{eqnarray*}
    a_{\cal P}(v_{\cal P},\Pi_{\cal P}v_{\cal P}) &\ge& \frac{\alpha_0}{2}|v_{\cal P}|_{1,\cal P}^2 +\frac{\kappa}{2}\|v_{\cal P}\|_{0,\cal P}^2\ge \frac12\min\{\alpha_0,\kappa\}\|v_{\cal P}\|_{1,\cal P}^2.
\end{eqnarray*}
Noticing \eqref{Poincare} and \eqref{equ-norm}, we obtain
\[
    \|\Pi_{\cal P}v_{\cal P}\|_{\cal P'} \lesssim  \|v_{\cal P}\|_{1,\cal P}.
\]
Therefore, for any $v_{\cal P}\in U^r_{\cal P}$,
\[
\sup_{w_{\cal P'}\in V_{\cal P'}} \frac{a_{\cal P}(v_{\cal P},w_{\cal P'})}{\|w_{\cal P'}\|_{\cal P'}}
\ge \frac{a_{\cal P}(v_{\cal P},\Pi_{\cal P}v_{\cal P})}{\|\Pi_{\cal P}v_{\cal P}\|_{\cal P'}}\ge c_0 \|v_{\cal P}\|_{1,\cal P},
\]
where $c_0$ is a constant depending only on $r,\alpha_0,\kappa$ and $\Omega$.
The inf-sup condition \eqref{infsup} follows.
\end{proof}

\begin{remark}
In the above proof, the partition ${\cal P}$ does not need to satisfy any quasi-uniform or shape-regularity
condition. Moreover,  \eqref{infsup} holds even the order of polynomials in each sub-interval $\tau_i$ are different.
\end{remark}

\subsection{Energy norm error estimate}
Following \cite{Xu.J.Zou.Q2009},  we
use the inf-sup condition \eqref{infsup} and the framework of Petrov-Galerkin method
to present and analyze the finite volume element method \eqref{bilinear}.

We first introduce a semi-norm and a norm in the  broken space
$
   H_{\cal P}^{2}(\Omega)= W^{2,2}_{\cal P}(\Omega)
$
  by
\[
  |v|^2_{\cal P}=\sum_{\tau_i\in\cal P}|v|_{1,\tau_i}^2+h_i^2|v|_{2,\tau_i}^2,\ \  \|v\|_{\cal P}^2 = \|v\|_{0,\cal P}^2+|v|_{\cal P}^2.
\]
 It is straightforward to show  that,
\[
  |v_{\cal P}|_{\cal P}\sim |v_{\cal P}|_{1,\cal P}, \qquad \|v_{\cal P}\|_{\cal P}\sim \|v_{\cal P}\|_{1,\cal P},\qquad \forall v_{\cal P}\in U_{\cal P}^r.
\]
With these equivalences,  the inf-sup condition \eqref{infsup} can be written as
\begin{equation}\label{infsup2}
\inf_{v_{\cal P}\in U_{\cal P}^r}\sup_{w_{\cal P'} \in V_{\cal P'} }
\frac{a_{\cal P} (v_{\cal P} ,w_{\cal P'} )}{\|v_{\cal P} \|_{\cal P} \|w_{\cal P'}\|_{\cal P'}}\ge c_2,
\end{equation}
where $c_2>0$ also depends only on $r,\alpha_0, \kappa$ and $\Omega$. Moreover, we define a discrete semi-norm $\big|\cdot\big|_{G,1}$ for all $v\in H_0^1(\Omega)$ by
\[
\big|v\big|_{G,1}=\left(\sum_{i=1}^N\sum_{j=1}^{r}A_{i,j}(v'(g_{i,j}))^2\right)^{\frac 12}.
\]

 We next discuss the relationship between $|\cdot|_{\cal P}$ and $|\cdot|_{G,1}$. First
\[
   |v_{\cal P}|_{G,1}= |v_{\cal P}|_{1,\cal P}\thicksim |v_{\cal P}|_{\cal P},\quad \forall v_{\cal P}\in U_{\cal P}^r.
\]
  On the other hand, for all $v\in H_0^{1}(\Omega)\cap H_{\cal P}^{2}(\Omega)$,
\[
   (v'(g_{i,j}))^2\lesssim  h^{-1}_{i}\|v'\|_{L^2(\tau_i)}^2+h_{i}\|v''\|_{L^2(\tau_i)}^2, \quad (i,j)\in \bZ_N \times \bZ_r,
\]
where the hidden constant depends only on $r$.
Thus by the fact  $A_{ij}\le h_i$, we have
\begin{eqnarray*}
  |v|_{G,1}^2&=&\sum_{i=1}^N\sum_{j=1}^r A_{ij}(v'(g_{i,j}))^2\\
  &\lesssim & \sum_{i=1}^N h_i\left(h^{-1}_{i}\|v'\|_{L^2(\tau_i)}^2+h_{i}\|v''\|_{L^2(\tau_i)}^2\right) =  |v|_{\cal P}^2.
\end{eqnarray*}
Namely,
\[
|v|_{G,1}\lesssim |v|_{\cal P},\quad \forall v\in H_0^1(\O) \cap H_{\cal P}^2(\O),
\]
where the hidden constant depends only on $r$.

\medskip

We are ready to show our main result.
\begin{theorem}
Assume that $u$ is the solution of \eqref{Poisson}, $u_{\cal P}$ is the solution of \eqref{bilinear}. Then the finite volume bilinear form $a_{\cal P}(\cdot,\cdot)$
 is variationally exact:
 \begin{equation}\label{exact}
a_{\cal P}(u,w_{\cal P'})=(f,w_{\cal P'}),\ \ \ \ \forall\  w_{\cal P'}\in \ V_{\cal P'},
  \end{equation}
 and bounded : for all\ $v\in\ H_0^1(\Omega)\cap H_{\cal P}^{2}(\Omega),\ w_{\cal P'}\in \ V_{\cal P'}$,
\begin{equation}\label{bounded}
|a_{\cal P}(v,w_{\cal P'})|\le M \|v\|_{\cal P} \|w_{\cal P'}\|_{\cal P'},
\end{equation}
  where the constant $M>0$  depends only on $r,\alpha_0,\kappa$ and $\O$.
Consequently,
\begin{equation}\label{totalestimate}
\big\|u-u_{\cal P}\big\|_{\cal P}\le \frac{M}{c_2} \inf_{v_{\cal P}\in U^r_{{\cal P}}}\big\|u-v_{\cal P}\big\|_{\cal P},
\end{equation}
where $c_2$ is the same as in \eqref{infsup2}.
\end{theorem}
\begin{proof} First, the formula \eqref{exact} follows  by multiplying \eqref{Poisson} with an arbitrary function $w_{\cal P'}\in V_{\cal P'} $
and then using Newton-Leibniz formula on each control volume $[g_{i,j},g_{i,j+1}], (i,j)\in \bZ_N \times \bZ_{r_i}$.

Next we show \eqref{bounded}. By the Cauchy-Schwartz inequality, from \eqref{bilinear1} there holds for all\ $v\in\ H_0^1(\Omega)\cap H_{\cal P}^{2}(\Omega),\ w_{\cal P'}\in \ V_{\cal P'}$ that
\begin{eqnarray*}
a_{\cal P}(v,w_{\cal P'})&\leq&|v|_{G,1}\left(\sum_{i=1}^{N}\sum_{j=1}^{r} \frac{\alpha^2(g_{i,j})}{A_{ij}}([w_{i,j}])^2\right)^\frac12 \\
 &+&\max(|\beta|,|\gamma|)\left(\sum_{i=1}^{N}\sum_{j=1}^{r_i}h_iw_{i,j}^2\right)^{\frac 12}\left(\sum_{i=1}^{N}(|v|_{1,\tau_i}^2+\|v\|_{0,\tau_i}^2)\right)^{\frac 12} \\
&\leq&M \|v\|_{\cal P} \|w_{\cal P'}\|_{\cal P'},
\end{eqnarray*}
where the constant $M$ depends only on $r,\alpha_0,\kappa$ and $\O$.

Finally, combining the inf-sup condition \eqref{infsup2}, \eqref{exact} and \eqref{bounded}, we derive \eqref{totalestimate}
 following similar arguments as in Babuska and Aziz (\cite{BabuskaAziz1972}), or Xu and Zikatanov
(\cite{XuZikatanov2003}).

\end{proof}
\begin{coro}
  Assume that $u\in H_0^1(\Omega)\cap H^{r+1}_{\cal P}(\Omega)$ is the solution of \eqref{Poisson}, and
  $u_{\cal P}$ is the solution of FVM scheme \eqref{bilinear}, then
\begin{equation}\label{optimalh1}
   \|u-u_{\cal P}\|_{1}\lesssim h^r|u|_{r+1,\cal P},
\end{equation}
  where the hidden constant is independent of  $\cal P$.
\end{coro}
\begin{proof}
   It follows from the definition of $\|\cdot\|_{\cal P}$ and \eqref{totalestimate} that
\[
  \|u-u_{\cal P}\|_1\le \|u-u_{\cal P}\|_{\cal P}\lesssim \inf\limits_{v_{\cal P}\in U^r_{\cal P}}\|u-v_{\cal P}\|_{\cal P}.
\]
  Noticing that
\[
   \inf\limits_{v_{\cal P}\in U^r_{\cal P}}\|u-v_{\cal P}\|_{\cal P}\le \|u-u_I\|_{1}+\left(\sum_{i=1}^{N}h_i^2|u-u_I|^2_{2,\tau_i}\right)^{\frac 12},
\]
where $u_I$ is the Lagrange interpolation of $u$ at the Lobatto points (defined in the next section)
in the trial space $U^r_{\cal P}$. By the standard approximation theory, we obtain the estimate \eqref{optimalh1}.
\end{proof}

\section{Superconvergence}
\setcounter{equation}{0}
In this section, we will present the superconvergence properties of the FVM solution. To this end, we need to use the interpolation of a function on the so-called  Lobatto points. This kind of interpolation has been used in the literature for superconvergence analysis, see, e.g., \cite{Zhang1, Zhang2}.
We denote $L_1,L_2,\cdots,L_{r-1}$ the zeros of $P'_r(x)$, where $P_r(x)$ is the Legendre polynomial of order $r$. Moreover, we denote $L_0=-1,L_r=1$ and $\bN_r=\{0,1,\cdots,r\}$ for $r\ge 1$. The family of points $L_j ,j\in \bN_r$ are called Lobatto points of degree $r$. The Lobatto points on $\tau_i$ are defined as the affine transformations of $L_j$ to $\tau_i$, i.e,
\[
    l_{i.j}=\frac12(x_i+x_{i-1}+h_iL_j),\quad j\in \bN_r.
\]
Let $u_I\in U^r_{\cal P}$ be the interpolation of $u$ such that
\[
   u_I(l_{i,j})=u(l_{i,j}),\quad (i,j)\in \bZ_N \times \bN_r,
\]
  then by \cite{Zhu.QD;LinQ1989}(p146, (1.2))
\begin{eqnarray}\label{est1}
    |(u-u_I)'(g_{i,j})|\lesssim h^{r}|u|_{r+2,1,\omega'_{i,j}},
\end{eqnarray}
where $\omega'_{i,j}=(g_{i,j-1},g_{i,j})$.

\begin{theorem}\label{h1estimate}
    Let $u\in H^1_0(\Omega)\cap H_{\cal P}^{r+2}(\Omega)$ be the solution of  \eqref{Poisson}, and $u_{\cal P}$ the solution of FVM scheme \eqref{bilinear}. Then
    \begin{eqnarray}\label{bilest}
        |a_{\cal P}(u-u_I,w_{\cal P'})| \lesssim h^{r+1}\left(|u|_{r+2,\cal P}+|u|_{r+1,\infty,\cal P}\right)\|w_{\cal P'}\|_{\cal P'},\quad \forall w_{\cal P'} \in V_{\cal P'}.
    \end{eqnarray}
    Consequently,
    \begin{eqnarray}\label{h1est}
        \|u_I-u_{\cal P}\|_1\lesssim h^{r+1}\left(|u|_{r+2,\cal P}+|u|_{r+1,\infty,\cal P}\right).
    \end{eqnarray}
\end{theorem}

\begin{proof} Recall the definition of bilinear form \eqref{bilinear1},  integral by part we obtain
\[
     a_{\cal P}(u-u_I,w_{\cal P'})=I_1+I_2
\]
 with
 \begin{eqnarray*}
     &&I_1=\sum_{i=1}^{N}\sum_{j=1}^{r}[w_{i,j}]\left(\alpha(g_{i,j})(u-u_I)'(g_{i,j})-\beta(g_{i,j})(u-u_I)(g_{i,j})\right),\\
      &&I_2=\sum_{i=1}^{N}\sum_{j=1}^{r}w_{i,j}\int_{g_{i,j}}^{g_{i,j+1}}(\gamma-\beta')(u-u_I).
 \end{eqnarray*}
 For all $i\in \bZ_{N}$, note that
\[
  (u-u_I)(g_{i,j})\lesssim h_i^{r+1}|u|_{r+1,\infty,\tau_i},\ \ |u|_{r+2,1,\tau_i} \lesssim h^{\frac12}_{i}|u|_{r+2,\tau_i},
\]
by the Cauchy-Schwartz inequality, \eqref{est1} and the standard approximation theory,
\begin{eqnarray*}
        a_{\cal P}(u-u_I,w_{\cal P'})&\lesssim& \|w_{\cal P'}\|_{\cal P'}\left( \sum_{i=1}^{N}\sum_{j=1}^{r}\left( h_i^{2(r+1)}|u|^2_{r+2,2,\tau_i}+h_i^{2r+3}|u|^2_{r+1,\infty,\tau_i}+\|u-u_I\|^2_{0,\tau_i}  \right) \right)^{\frac12}\\
        &\lesssim& h^{r+1}\left(|u|_{r+2,\cal P}+|u|_{r+1,\infty,\cal P}\right)\|w_{\cal P'}\|_{\cal P'}.
\end{eqnarray*}
    The desired result \eqref{bilest} is proved.

    We next show \eqref{h1est}. By the inf-sup condition \eqref{infsup},
    \[
    c_0 \|u_I-u_{\cal P}\|_{1}\leq \sup_{w_{\cal P'}\in V_{\cal P'}}\frac{a_{\cal P}(u_{\cal P}-u_I,w_{\cal P'})}{\|w_{\cal P'}\|_{\cal P'}}=\sup_{w_{\cal P'}\in V_{\cal P'}}\frac{a_{\cal P}(u-u_I,w_{\cal P'})}{\|w_{\cal P'}\|_{\cal P'}}.
    \]
 Combining the above inequality with \eqref{bilest}, we derive \eqref{h1est}.
\end{proof}

As a direct consequence of \eqref{h1est}, we have the following $L^2$ error estimate.
\begin{corollary}\label{l2estimate}
Let $u\in H^1_0(\Omega)\cap H^{r+2}_{\cal P}(\Omega)$ be the solution of  \eqref{Poisson}, and $u_{\cal P}$ the solution of FVM scheme \eqref{bilinear}, then
    \begin{eqnarray}\label{optimall2}
        \|u-u_{\cal P}\|_0 \lesssim h^{r+1}\|u\|_{r+2,\cal P},
    \end{eqnarray}
 where the hidden positive constant is independent of $\cal P$.
 \end{corollary}
\begin{proof}
    By the triangle inequality,
    \[
        \|u-u_{\cal P}\|_0\leq \|u-u_I\|_0+\|u_I-u_{\cal P}\|_0.
    \]
    By the Poincar\'{e} inequality and \eqref{h1est}, we have
    \[
        \|u_I-u_{\cal P}\|_0 \lesssim |u_I-u_{\cal P}|_1\lesssim h^{r+1}\|u\|_{r+2,\cal P}.
    \]
    Moreover, by the standard approximation theory,
    \[
        \|u-u_I\|_0 \lesssim h^{r+1}\|u\|_{r+1} \lesssim h^{r+1}\|u\|_{r+2,\cal P}.
    \]
The desired estimate \eqref{optimall2} follows.
\end{proof}
\begin{remark}
In the above $L^2$ error estimate, we do not use the so-called Aubin-Nitsche technique.
However, we need slightly stronger regularity assumption on the exact solution $u$.
\end{remark}
\bigskip

We next study the superconvergence property at the nodes $x_i, i\in \bZ_{N-1}$.
\begin{theorem}\label{supconv_vectices}
     Let $u$ be the solution of  \eqref{Poisson}, and $u_{\cal P}$ the solution of FVM scheme \eqref{bilinear}. If $u\in W^{2r+1,\infty}_{\cal P}(\Omega)$, then
     \begin{eqnarray}\label{e_vertice1}
        |(u-u_{\cal P})(x_{i})|\lesssim h^{2r} \sum_{j=r+1}^{2r+1} |u|_{j,\infty,\cal P},\ \ \forall i\in \bZ_{N-1}.
    \end{eqnarray}
\end{theorem}
\begin{proof}Let $e=u-u_{\cal P}$ and
\[
\e(x)=\int_a^x \big( \beta(y)e'(y)+\gamma(y)e(y) \big) dy,\quad \forall x\in [a,b].
\]
By the construction of the FVM scheme, both $u$ and $u_{\cal P}$ satisfy \eqref{conserve}, then
 for all $(i,j)\in \bZ_{N-1}\times \bZ_{r_i}$,
\begin{eqnarray*}
-\big( \alpha(g_{i,j+1})e'(g_{i,j+1})-\alpha(g_{i,j})e'(g_{i,j}) \big)+\e(g_{i,j+1})-\e(g_{i,j})=0.
\end{eqnarray*}
Namely,
\begin{equation}\label{est4}
\alpha(g_{i,j})e'(g_{i,j})-\e(g_{i,j})=C_0,
\end{equation}
where $C_0$ is a constant independent of $i,j$.

On the other hand, let $G(\cdot,\cdot)$ be the Green function for the problem \eqref{Poisson}. Then for all $v\in H_0^1(\Omega)$,
\[
v(x)=\int_a^b\alpha(y)v'(y)\frac{\partial G}{\partial y}(x,y) dy+\int_a^b \big( \beta(y)v'(y)+\gamma(y)v(y) \big)G(x,y)dy.
\]
In particular, for all $i \in \bZ_{N-1}$,
\[
e(x_i)=\int_a^b \alpha(y)e'(y)\frac{\partial G}{\partial y}(x_i,y) dy+\int_a^b \big( \beta(y)e'(y)+\gamma(y)e(y) \big)G(x_i,y)dy.
\]
Noting that $G(x_i,a)=G(x_i,b)=0$, by\eqref{est4}
\begin{eqnarray*}
e(x_i)&=&\int_a^b \big(\alpha(y)e'(y)-\e(y)\big)\frac{\partial G}{\partial y}(x_i,y) dy\\
      &=&\sum_{k=1}^{N}\sum_{j=1}^{r} A_{k,j}\big(\alpha(g_{k,j})e'(g_{k,j})-\e(g_{k,j})\big)\frac{\partial G}{\partial y}(x_i,g_{k,j})+E_1\\
      &=&C_0\int_a^b\frac{\partial G}{\partial y}(x_i,y)dy +E_1+E_2=E_1+E_2,
\end{eqnarray*}
where
\begin{eqnarray*}
E_1&=&\left.\sum_{k=1}^N \frac{h_k^{2r+1}(r!)^4}{(2r+1)[(2r)!]^3}\left[\big( (\alpha(y)e'(y)-\e(y)\big)\frac{\partial G}{\partial y}(x_i,y)\right]^{(2r)}\right |_{y=\xi_k},\\
E_2&=&\left.-\sum_{k=1}^N \frac{h_k^{2r+1}(r!)^4}{(2r+1)[(2r)!]^3}\left[\frac{\partial G}{\partial y}(x_i,y) \right]^{(2r)}\right|_{y=\eta_k}
\end{eqnarray*}
with $\xi_k,\eta_k \in \tau_k$.

We next estimate $E_1$ and $E_2$, respectively.
Note that $e^{(j)}=u^{(j)}$ for $j>r$ and the Green function $G(x_i,\cdot)$ has bounded derivatives of any order on each $\tau_k, k\in \bZ_{N}$, then
 \begin{eqnarray*}
    |E_1| &\lesssim&\sum_{k=1}^N h_k^{2r+1}\left( \sum_{j=0}^r |e|_{j,\infty,\tau_k}+\sum_{j=r+1}^{2r+1} |u|_{j,\infty,\tau_k}\right)\\
          &\lesssim &\sum_{k=1}^{N}h_k^{2r+1}\left(\sum_{j=0}^{r}h^{-j}\|e\|_{0,\infty,\tau_k}+\sum_{j=r+1}^{2r+1} |u|_{j,\infty,\tau_k}\right)\\
          &\lesssim &h^{r}\|e\|_{0,\infty,\cal P}+h^{2r+1}\sum_{j=r+1}^{2r+1} |u|_{j,\infty,\cal P},
 \end{eqnarray*}
  where in the second inequality we have used the fact that\cite{adams}
\[
    |e|_{j,\infty,\tau_k} \lesssim h^{-j}_k\|e\|_{0,\infty,\tau_k}+h^{r+1-j}_k|e|_{r+1,\infty,\tau_k},\ \ \forall j\in\bZ_{r}.
\]
 We next consider the term $\|e\|_{0,\infty,\cal P}$. By the triangular inequality and the inverse inequality
\begin{eqnarray*}
    \|u_I-u_{\cal P}\|_{0,\infty,\cal P}& \lesssim& h^{\frac12}|u_I-u_{\cal P}|_{1}\lesssim h^{\frac12}(|u-u_I|_1+|u_I-u_{\cal P}|_1)\\
    &\lesssim& h^{r+\frac12}|u|_{r+1}\lesssim h^{r+\frac12}|u|_{r+1,\infty,\cal P},
\end{eqnarray*}
  we have
\[
   \|e\|_{0,\infty,\cal P}\le \|u-u_I\|_{0,\infty,\cal P}+\|u_I-u_{\cal P}\|_{0,\infty,\cal P}\lesssim h^{r+\frac 12}|u|_{r+1,\infty,\cal P}.
\]
Therefore,
\begin{eqnarray*}
    |E_1|  \lesssim  h^{2r+\frac 12} \sum_{j=r+1}^{2r+1}|u|_{j,\infty,\cal P}.
\end{eqnarray*}
 As for $E_2$,  a direct calculation shows that
\[
   |E_2|  \lesssim  h^{2r}\|G\|_{2r,\infty,\cal P}.
\]
Combining  $E_1$ with $E_2$, we obtain \eqref{e_vertice1}.
\end{proof}

As a direct consequence of \eqref{e_vertice1}, we have
\begin{equation}\label{avg-nod}
E_{node}=\left(\frac{1}{N}\sum_{i=1}^{N} [ (u-u_{\cal P})(x_i) ]^2\right)^\frac{1}{2} \lesssim h^{2r}.
 \end{equation}

   We next estimate the term $e(x_i)-e(x_{i-1}), i\in\bZ_{N}$ which plays a critical role in our later superconvergence analysis.
\begin{theorem}\label{supconv_vectrices2} For all $i\in\bZ_N$,
\begin{equation}\label{supconv_vectices1}
   | (u-u_{\cal P})(x_i)-(u-u_{\cal P})(x_{i-1})|\lesssim h^{2r+1} \sum_{j=r+1}^{2r+1}|u|_{j,\infty,\cal P}.
\end{equation}
\end{theorem}

\begin{proof}
 By the same arguments as Theorem \ref{supconv_vectices}, we obtain
\[
        e(x_i)-e(x_{i-1})=\sum_{k=1}^{N}(E'_{1,k}+E'_{2,k}),
\]
where
\begin{eqnarray*}
E'_{1,k}&=&\left. \frac{h_k^{2r+1}(r!)^4}{(2r+1)[(2r)!]^3}\left[\big( (\alpha(y)e'(y)-\e(y)\big)\left( \frac{\partial G}{\partial y}(x_i,y) -\frac{\partial G}{\partial y}(x_{i-1},y) \right)\right]^{(2r)}\right|_{y=\xi_k},\\
E'_{2,k}&=&\left.-\frac{h_k^{2r+1}(r!)^4}{(2r+1)[(2r)!]^3}\left( \frac{\partial G}{\partial y}(x_i,y) -\frac{\partial G}{\partial y}(x_{i-1},y) \right)^{(2r)}\right|_{y=\eta_k}
\end{eqnarray*}
with  $\xi_k,\eta_k \in \tau_k$.\\
Recall the construction of the Green function $G(x_i,\cdot)$, for all $j\in \bN_{2r}$,
\[
   \| \frac{\partial G}{\partial y}(x_i,y) -\frac{\partial G}{\partial y}(x_{i-1},y)\|_{j,\infty,\Omega\setminus\tau_{i}}\lesssim h\|G\|_{j+1,\infty,\Omega\setminus\tau_{i}},
\]
 and
\[
   \| \frac{\partial G}{\partial y}(x_i,y) -\frac{\partial G}{\partial y}(x_{i-1},y)\|_{j,\infty,\tau_{i}}\lesssim h\|G\|_{j+1,\infty,\tau_{i}}.
\]
 Since the Green function $G(x_i,\cdot)\in C^{2r}(\tau_k), k\in \bZ_{N}$ is bounded, then
\begin{equation*}
 \begin{split}
   &\left.\left[\big( (\alpha(y) e'(y)-\e(u)\big)\left( \frac{\partial G}{\partial y}(x_i,y) -\frac{\partial G}{\partial y}(x_{i-1},y) \right)\right]^{(2r)}\right|_{y=\xi_k}\\
   &\lesssim \sum_{j=0}^{2r}\binom{2r}{j}\|\alpha e'-\e\|_{j,\infty,\tau_k}\| \frac{\partial G}{\partial y}(x_i,y) -\frac{\partial G}{\partial y}(x_{i-1},y)\|_{2r-j,\infty,\tau_{k}}\\
   &\lesssim  h_k\left( \sum_{j=0}^{r}|e|_{j,\infty,\tau_k}+\sum_{j=r+1}^{2r+1}|u|_{j,\infty,\tau_k}\right).
 \end{split}
\end{equation*}
  Following the same estimate for $\sum_{j=0}^{r}|e|_{j,\infty,\tau_k}+\sum_{j=r+1}^{2r+1}|u|_{j,\infty,\tau_k}$ as Theorem \ref{supconv_vectices}, we obtain
\[
  | E'_{1,k}|\lesssim h_k^{2r+2},\ \ |E'_{2,k}|\lesssim h_k^{2r+2},\ \ \forall k\in \bZ_{N},
\]
 which yields the inequality \eqref{supconv_vectices1} directly.
\end{proof}

\bigskip

Next we present the superconvergence property of $u_{\cal P}'$ at Gauss points, and $u_{\cal P}$ at Lobatto points.
Before our analysis, we first introduce a special polynomial.  For all $v(t)\in H^1([-1,1])$, we denote by
\[
     v_r(t)=\sum_{j=0}^rb_jM_j(t)
\]
  the $r$th approximation of $v(t)$ with
\begin{eqnarray*}
    && b_0=\frac{v(1)+v(-1)}{2},\ b_1=\frac{v(1)-v(-1)}{2},\\
    && b_j=(j-\frac 12)\int_{-1}^{1}v'(t)L_{j-1}(t) dt,\ \ j=2,\ldots,r.
\end{eqnarray*}
  where $M_i$ is the Lobatto polynomial of degree $i$ and $L_{j}$ is the Legendre polynomial of degree $j$.
  For all $x\in\tau_i, i\in\bZ_N$, we denote by
\[
   v_r(x)=v_r(\frac{x_i+x_{i-1}+h_it}{2}),\ \ t\in[-1,1]
\]
   the $r$th approximation of $v(x)$ on the interval $\tau_i, i\in\bZ_N$. Then
\begin{equation}\label{l1}
  |(v-v_r)(l_{i,j})|\lesssim h^{r+2}\|v\|_{r+2,\infty,\cal P},\ \ j\in\bZ_{r-1},
\end{equation}
  and
\begin{equation}\label{g2}
  |(v-v_r)'(g_{i,j})|\lesssim h^{r+1}\|v\|_{r+2,\infty,\cal P},\ \ j\in\bZ_r.
\end{equation}

\begin{theorem}\label{supconv_gauss}
     Let $u\in W^{r+2,\infty}_{\cal P}(\Omega)$ be the solution of  \eqref{Poisson}, and $u_{\cal P}$ the solution of FVM scheme \eqref{bilinear}.
     Then
    \begin{eqnarray}\label{e_gauss2}
        |(u-u_{\cal P})'(g_{i,j})|\lesssim h^{r+1}\|u\|_{r+2,\infty,\cal P},\ \  (i,j)\in \bZ_N \times \bZ_{r-1},
    \end{eqnarray}
    and
    \begin{eqnarray}\label{e_Lobatto2}
        |(u-u_{\cal P})(l_{i,j})|\lesssim h^{r+2}\|u\|_{r+2,\infty,\cal P},\ \ (i,j)\in \bZ_N \times \bZ_r.
    \end{eqnarray}
\end{theorem}
\begin{proof}For all $v\in H_0^1(\Omega)$, let
\[
   A(u,v)=(\alpha u',v')+(\beta u'+\gamma u,v).
\]
  Then we have
\[
   v(x)=A(v,G(x,\cdot)),\ \ \forall x\in \Omega,
 \]
  where $G(x,\cdot)$ is the Green function for the problem \eqref{Poisson}.
  Let $G_{\cal P}$ be the Garlerkin approximation of $G(x,\cdot)$, that is
\[
   v_{\cal P}(x)=A(v_{\cal P},G_{\cal P}),\ \ \forall v_{\cal P}\in U_{\cal P}^r, \forall x\in\Omega.
\]
  Then(see\cite{Chen.C.M2001}(p33))
\[
  A(u-u_r,G_{\cal P})\lesssim h^{r+2}\|u\|_{r+2,\infty,\cal P}.
\]
  We next estimate the term $A(u-u_{\cal P},G_{\cal P})$. Note that $G_{\cal P}\in U_{\cal P}^r$, then
\begin{eqnarray*}
A(u-u_{\cal P},G_{\cal P})&=&\int_a^b \big(\alpha(y)e'(y)-\e(y)\big)G'_{\cal P}(y) dy\\
      &=&\sum_{k=1}^{N}\sum_{j=1}^{r} A_{k,j}\big(\alpha(g_{k,j})e'(g_{k,j})-\e(g_{k,j})\big)G_{\cal P}(g_{k,j})+E_3\\
      &=&C_0\int_a^b G'_{\cal P} dy +E_3=E_3,
\end{eqnarray*}
  where $e(y),\e(y)$ and $C_0$ are the same as in Theorem \ref{supconv_vectices} and
\[
  E_3=\left.\sum_{k=1}^N \frac{h_k^{2r+1}(r!)^4}{(2r+1)[(2r)!]^3}\left[\big(\alpha(y) e'(y)-\e(y)\big)G'_{\cal P}(y)\right]^{(2r)}\right |_{y=\xi_k}.
\]
  Then
\begin{eqnarray*}
   E_3 &\lesssim&\sum_{k=1}^N h_k^{2r+1}\left(|G_{\cal P}|_{r,\infty,\tau_k}\|e\|_{r+2,\infty,\tau_k}+\sum_{j=1}^{r-1}|G_{\cal P}|_{j,\infty,\tau_k}\|e\|_{2r-j+2,\infty,\tau_k}\right)\\
       &\lesssim&\sum_{k=1}^N |G_{\cal P}|_{2,1,\tau_k}\left(h_k^{r+2}\|e\|_{r+2,\infty,\tau_k}+\sum_{j=1}^{r-1}h_k^{2r+2-j}\|e\|_{2r-j+2,\infty,\tau_k}\right)\\
       &\lesssim& h^{r+2}\|e\|_{r+2,\infty,\cal P}\lesssim h^{r+2}\|u\|_{r+2,\infty,\cal P}.
\end{eqnarray*}
  Here we have used \eqref{e1_1} and the inverse inequality
\[
   |G_{\cal P}|_{j,\infty,\tau_k}\lesssim h_k^{1-j}|G_{\cal P}|_{2,1,\tau_k}, \ \forall j\in\bZ_{r}
\]
 and the fact \cite{Chen.C.M2001}(p33)
\[
   |G_{\cal P}|_{2,1,\cal P}=\sum_{i=1}^{N}|G_{\cal P}|_{2,1,\tau_k}\le C
\]
 with $C$ a bounded constance.
 Note that
\[
   (u_r-u_{\cal P})(x)=A(u_r-u_{\cal P},G_{\cal P})=A(u-u_r,G_{\cal P})+A(u-u_{\cal P},G_{\cal P}),
\]
 then we have
\[
   (u_r-u_{\cal P})(x)\lesssim h^{r+2}\|u\|_{r+2,\infty,\cal P}.
\]
 By inverse inequality,
\[
   (u_r-u_{\cal P})'(x)\lesssim h^{r+1}\|u\|_{r+2,\infty,\cal P}.
\]
  Combing above estimates with \eqref{l1} and \eqref{g2}, we obtain
  \eqref{e_gauss2} and \eqref{e_Lobatto2} directly by the triangular inequality.


\end{proof}

As a direct consequence , we have
 \begin{equation}\label{sup_guass}
 |u-u_{\cal P}|_{G,1} \lesssim h^{r+1}, \quad |u-u_{\cal P}|_{aver,1} \lesssim h^{r+1}
\end{equation}
 and
\begin{equation}\label{sup}
 |u-u_{\cal P}|_{L,0} \lesssim h^{r+2}, \quad |u-u_{\cal P}|_{aver,0} \lesssim h^{r+2},
\end{equation}
where
\[
    |v|_{{\rm aver},1} = \left(\frac{1}{Nr}\sum_{i=1}^N \sum_{j=1}^r v'(g_{i,j})^2\right)^\frac12
\]
 and
\begin{equation*}\label{semi-norm}
 |v|_{L,0} = \left(\sum_{i=1}^N \sum_{j=0}^r w_{i,j} v(l_{i,j})^2\right)^\frac12,\quad  |v|_{{\rm aver},0} = \left(\frac{1}{Nr}\sum_{i=1}^N \sum_{j=0}^r v(l_{i,j})^2\right)^\frac12,
\end{equation*}
 here $w_{i,j}$ are weights of the Lobatto quadrature.
\bigskip

Now we consider a special case that $\beta=0$, we have the following theorem.
\begin{theorem}\label{supconv_gauss1}
     Let $u$ be the solution of  \eqref{Poisson}, and $u_{\cal P}$ the solution of FVM scheme \eqref{bilinear}. If $\beta(x)=0,\forall x\in \O, u\in W^{r+2,\infty}_{\cal P}(\Omega)$, then
     \begin{eqnarray}\label{e_gauss8}
        |(u-u_{\cal P})'(g_{i,j})|\lesssim h^{\min\{r+2,2r\}}\sum_{k=r+1}^{2r+1}|u|_{k,\infty,\cal P}.
    \end{eqnarray}
\end{theorem}
\begin{proof}
First, both $u$ and $u_{\cal P}$  satisfy \eqref{conserve}, there holds for all $ (i,j)\in \bZ_N\times \bZ_{r-1}$ that
\begin{eqnarray*}
    \alpha(g_{i,j+1})e'(g_{i,j+1})-\alpha(g_{i,j})e'(g_{i,j})=\int_{g_{i,j}}^{g_{i,j+1}} \gamma(x)e(x)dx,
\end{eqnarray*}
     which yields
 \begin{equation}\label{est2}
     e'(g_{i,j+1})=\frac{\alpha(g_{i,1})}{\alpha(g_{i,j+1})}e'(g_{i,1})+\frac{1}{\alpha(g_{i,j+1})}\int_{g_{i,1}}^{g_{i,j+1}} \gamma(x)e(x)dx.
 \end{equation}
 On the other hand,
 \[
       e(x_i)-e(x_{i-1})=\int_{x_{i-1}}^{x_i} e'(y)dy=  \sum_{j=1}^r A_{i,j}e'(g_{i,j})+E_i,
\]
where by \cite{DavisRabinowitz1984}(p98, (2.7.12)),
\begin{eqnarray*}
    E_i=\frac{h_i^{2r+1}(r!)^4}{(2r+1)[(2r)!]^3}(e')^{(2r)}(\xi_i)
    \lesssim h^{2r+1}|u|_{2r+1,\infty,\tau_i},\quad \xi_i \in \tau_i.
\end{eqnarray*}
   By Theorem \eqref{supconv_vectrices2} and \eqref{est2}, we obtain
\[
   h_ie'(g_{i,1})+h_i\int_{x_{i-1}}^{x_i}\gamma(x)e(x) dx\lesssim h_i^{2r+1}\sum_{k=r+1}^{2r+1}|u|_{k,\infty,\cal P}.
\]
   Then
\[
   e'(g_{i,j})\lesssim \int_{x_{i-1}}^{x_i}|\gamma(x)e(x)| dx+h_i^{2r}\sum_{k=r+1}^{2r+1}|u|_{k,\infty,\cal P},\ \ \forall j\in \bZ_{r}.
\]

 We next estimate the term $\int_{x_{i-1}}^{x_i}|\gamma(x)e(x)| dx$. Note that
 \begin{eqnarray*}
    \int_{g_{i,1}}^{g_{i,j}} \left|\gamma(x)e(x) \right| dx
    &\leq & \|\gamma\|_{\infty} \left( \int_{g_{i,1}}^{g_{i,j}} \left|(u-u_I)(x)\right| dx+\int_{g_{i,1}}^{g_{i,j}} \left|(u_I-u_{\cal P})(x) \right|dx \right)\\
    &\lesssim & h^{r+2}|u|_{r+2,\infty,\cal P} +h^{r+{5\over 2}}|u|_{r+2,\cal P} \lesssim h^{r+2}|u|_{r+2,\infty,\cal P},
\end{eqnarray*}
where in the above inequalities we have used $|u-u_I|\lesssim h^{r+1}\|u\|_{r+1,\infty}\lesssim h^{r+1}|u|_{r+2,\infty,\cal P}$ and
 $|u_I-u_{\cal P}|\lesssim h^{\frac 12}|u_I-u_{\cal P}|_{1,\cal P}\lesssim h^{r+\frac 32}|u|_{r+2,\cal P}$.

 Therefore,
\[
   |(u-u_{\cal P})'(g_{i,j})|\lesssim h^{\min\{r+2,2r\}}\sum_{k=r+1}^{2r+1}|u|_{k,\infty,\cal P}.
\]
The proof is completed.
 \end{proof}

\bigskip
In particular, when $\beta=\gamma=0$, we have a better result.

\begin{theorem}\label{Theo_5}Let $u\in W^{2r+1,\infty}_{\cal P}(\O)$ be the solution of \eqref{Poisson}, and
  $u_{\cal P}\in U^r_{\cal P}$ the solution of FVM scheme \eqref{bilinear}. If $\beta(x)=\gamma(x)=0,\forall x\in \O$, for all $(i,j)\in \bZ_N\times\bZ_r$, we have
 \begin{equation}\label{point_conv}
    |u'(g_{i,j})-u'_{\cal P}(g_{i,j})| \lesssim  h^{2r}\sum_{k=r+1}^{2r+1}|u|_{k,\infty,{\cal P}}.
 \end{equation}
\end{theorem}
\begin{proof}By \eqref{est4}, we denote the constant
\[
C=\alpha(g_{i,j})(u_{\cal P}'(g_{i,j})-u'(g_{i,j})).
\]
The fact $u_{\cal P}\in {\mathbb P}_{r}$ yields that
\begin{eqnarray*}\label{errnode}
e(x_i) &-& e(x_{i-1}) = \int_{x_{i-1}}^{x_i} e(t) dt \nonumber \\
&=&  \sum_{k=1}^r A_{i,k} e'(g_{i,k})  + \int_{x_{i-1}}^x u'(t) dt - \sum_{k=1}^r A_{i,k} u'(g_{i,k}).
\end{eqnarray*}
By \cite{DavisRabinowitz1984}(p98, (2.7.12))
\[
\left|\int_{x_{i-1}}^{x_i} u'(t) dt - \sum_{k=1}^r A_{i,k} u_{\cal P}'(g_{i,k})\right|
\lesssim h_i^{2r+1} |u|_{2r+1,\infty,\tau_i}
\]
and Theorem \ref{supconv_vectrices2}, we have
\[
    C\sum_{j=1}^r A_{i,j}\alpha^{-1}(g_{i,j}) \lesssim h_i^{2r+1}\sum_{k=r+1}^{2r+1}|u|_{k,\infty,\cal P},
\]
which yields \eqref{point_conv} directly.
\end{proof}

\begin{remark}
 We see that at the Gauss points, when $\beta=\gamma=0$, the derivative convergence rate $h^{2r}$ doubles the global optimal rate $h^r$,
 which is much better than the counterpart finite element method's $h^{r+1}$ rate, when $\beta=0$, the derivative convergence rate $h^{r+2}$ is one order higher than the counterpart finite element method's $h^{r+1}$;  and at the nodal points, the convergence rate $h^{2r}$ almost doubles the global optimal rate $h^{r+1}$ and equals to the counterpart finite element method's $h^{2r}$ rate; and at the  Lobatto points, the convergence rate $h^{r+2}$ is one order higher than the optimal global rate $h^{r+1}$,
 which is the same as the counterpart finite element method.
\end{remark}

\section{Post processing}
We observe from \eqref{e_gauss2}, \eqref{e_gauss8} and \eqref{point_conv} that $u'_{\cal P}$ approximates  the derivative of the exact solution $u$
pretty well at the Gauss points.  In this subsection, we will recover  $u'$  in the whole domain $\O$.

For all $i=1,\ldots, N-1$,  we construct a function $v_i\in \bP_{2r-1}([x_{i-1},x_{i+1}])$ by letting
\[
v_i(g_{l,k}) = u_{\cal P}'(g_{l,k}), \quad l=i,i+1; \;  k=1,2,\ldots,r.
\]
Then we define for all $x\in \tau_i= [x_{i-1},x_i], i=1,\ldots,N$,
\[
v(x)=\left\{
\begin{array}{lll}
v_1(x), &i=1,\\
\frac{1}{2}\big(v_i(x)+v_{i-1}(x)\big),  &2\le i\le N-1,\\
v_{N-1}(x),&i=N.
\end{array}
\right.
\]
To study the approximation property of $u$, we note that in each $[x_{i-1},x_{i+1}]$,
\[
u'(x)=(L_{2r-1}u')(x)+\frac{u^{(2r+1)}(\xi)}{(2r)!}\prod_{j=1}^r (x-g_{i,j})(x-g_{i+1,j}), \xi\in[x_{i-1},x_{i+1}]
\]
where the Lagrange interpolant
\[
(L_{2r-1}u' )(x)=\sum_{l=i}^{i+1}\sum_{j=1}^r u'(g_{l,j})w_{l,j}(x), w_{l,j}(x)=\prod_{l'\not=l,j'\not=j}\frac{x-g_{l',j'}}{g_{l,j}-g_{l',j'}}.
\]
Noting that
\[
v_{i}(x)=\sum_{l=i}^{i+1}\sum_{j=1}^r u'_{\cal P}(g_{l,j})w_{l,j}(x),
\]
we have
\begin{equation}\label{difference}
u'(x)-v_i(x)=\sum_{l=i}^{i+1}\sum_{j=1}^r (u'-u'_{\cal P})(g_{l,j})w_{l,j}(x)+\frac{u^{(2r+1)}(\xi)}{(2r)!}\prod_{j=1}^r (x-g_{i,j})(x-g_{i+1,j}).
\end{equation}
Since for all $l=i,i+1, j=1,\ldots, r$, we have
\[
|w_{l,j}(x)|\le c_r, \forall x\in [x_{i-1},x_{i+1}],
\]
where $c_r$ is a constant depends only on $r$, we obtain by \eqref{e_gauss2}, \eqref{e_gauss8} and \eqref{point_conv} that
\[
|u'(x)-v_i(x)|\lesssim h^m\sum_{k=r+1}^{2r+1}|u|_{k,\infty,{\cal P}},
\]
where $m=r+1$ for general elliptic equations, $m=r+2$ if $\beta=0$, and $m=2r$ if $\beta=\gamma=0$.
Consequently, we have
\[
|u'(x)-v(x)|\lesssim \lesssim h^m\sum_{k=r+1}^{2r+1}|u|_{k,\infty,{\cal P}}, \forall x\in\Omega.
\]
\section{Numerical experiments}
\setcounter{equation}{0}
In this section, we present numerical examples to demonstrate the method and to verify the theoretical results proved in this paper.

In our experiments, we solve the two-point boundary value problem \eqref{Poisson} by the FVM scheme \eqref{bilinear} with $r=4$ or $r=5$.
The underlying meshes are obtained by subdividing $\Omega=(0,1)$ to $N=2,4,8,16,32,64$ subintervals with equal sizes.

{\it Example} 1. We consider the two-point boundary value problem \eqref{Poisson} with
\[
        \alpha(x)=e^{x},\ \ \beta(x)=\cos x,\ \ \gamma(x)=x,\ \ \forall x\in\Omega,
\]
 and $f$ is chosen so that the exact solution of this problem is
\[
    u(x)=\sin x(x^{12}-x^{11}).
\]

We list approximate errors under various (semi-)norms in Table \ref{r=4} ( for the scheme $r=4$ ) and Table \ref{r=5} ( for the scheme $r=5$ ).

%

\begin{table}[htbp]\caption{$r=4$ \label{r=4} }
\centering
\begin{tabular}{|c||c|c|c|c|c|c|}
\hline
N & $\|u-u_{\cal P}\|_0$ &  $\|u-u_{\cal P}\|_1$ & $|u_I-u_{\cal P}|_1$ & $|u-u_{\cal P}|_{L,0}$ \\
\hline
2 &1.8618e-03 & 5.1201e-02& 5.2554e-03& 3.3420e-04\\
4 & 1.4386e-04& 7.2801e-03& 3.1271e-04& 9.8931e-06\\
8 & 5.9282e-06& 5.9099e-04& 1.1758e-05& 1.8624e-07\\
16& 1.9882e-07& 3.9516e-05& 3.8485e-07& 3.0490e-09 \\
32& 6.3240e-09& 2.5119e-06& 1.2166e-08& 4.8197e-11\\
64& 1.9850e-10& 1.5766e-07& 3.8129e-10& 7.5536e-13\\
\hline
\hline
 N & $|u-u_{\cal P}|_{aver,0}$ &$|u-u_{\cal P}|_{G,1}$
&$|u-u_{\cal P}|_{aver,1}$  & $E_{node}$ \\
\hline
2 & 2.1895e-04& 8.0770e-04& 5.4962e-04& 1.1874e-05\\
4 & 6.2680e-06& 5.3025e-05& 3.5877e-05& 5.9186e-08\\
8 & 1.1716e-07& 1.9692e-06& 1.3338e-06& 2.3666e-10\\
16& 1.9150e-09& 6.3947e-08& 4.3328e-08& 9.2827e-13\\
32& 3.0260e-11& 2.0170e-09& 1.3667e-09& ---\\
64& 4.7425e-13& 6.3175e-11& 4.2809e-11& ---\\
\hline
\end{tabular}
\end{table}

\begin{table}[htbp]\caption{$r=5$ \label{r=5} }
\centering
\begin{tabular}{|c||c|c|c|c|c|c|}
\hline N & $\|u-u_{\cal P}\|_0$ & $\|u-u_{\cal P}\|_1$
&$|u_I-u_{\cal P}|_1$ & $|u-u_p|_{L,0}$\\
\hline
2 &  4.8206e-04& 1.5546e-02& 8.5017e-04& 4.0891e-05    \\
4 &  1.5627e-05& 9.6503e-04& 2.1627e-05& 5.2643e-07\\
8 &  2.9713e-07& 3.6434e-05& 3.8065e-07& 4.6413e-09\\
16&  4.8711e-09& 1.1927e-06& 6.1190e-09& 3.7318e-11\\
32&  7.7022e-11& 3.7707e-08& 9.6282e-11& 2.9365e-13\\
64&  1.2073e-12& 1.1817e-09& 1.5081e-12& ---\\
\hline
\hline
N & $|u-u_{\cal P}|_{aver,0}$ &$|u-u_{\cal P}|_{G,1}$ &
$|u-u_{\cal P}|_{aver,1}$  & $E_{node}$ \\ \hline
2 & 2.6075e-05& 2.2179e-04 & 1.4819e-04& 4.6819e-08   \\
4 & 3.2965e-07& 5.8493e-06 & 3.9162e-06& 3.0508e-11\\
8 & 2.8971e-09& 1.0085e-07 & 6.7553e-08& 2.6318e-14\\
16& 2.3277e-11& 1.6089e-09 & 1.0779e-09& ---\\
32& 1.8311e-13& 2.5266e-11 & 1.6928e-11& ---\\
64& ---       & 3.9473e-13 & 2.6482e-13& ---\\
\hline
\end{tabular}
\end{table}

To explicitly show the convergence rate of different approximate errors, we plot the error curves in Figures \ref{1_1} and \ref{1_2}.
We observe from Figure \ref{1_1} that the convergence rate  $|u-u_{\cal P}|_{1}$ is $r$ and the convergence rate of $\|u-u_{\cal P}\|_0$ is $r+1$. In other words, the FVM approximate solution converges to
 the exact solution with optimal convergence rates under both
for $H^1$ and $L^2$ norms, as predicted in \eqref{optimalh1} and \eqref{optimall2}.
We  also observe that the error $|u_I-u_{\cal P}|_1$ is of order $r+1$, which confirms the convergence result in \eqref{h1est}.
The errors  $|u-u_{\cal P}|_{aver,0}$, $|u-u_{\cal P}|_{L,0}$ and $E_{node}$  are presented in Figure \ref{1_2}. It is observed that
 $|u-u_{\cal P}|_{aver,0}$ and $|u-u_{\cal P}|_{L,0}$ converge with a degree $r+2$ which confirm the superconvergence property at Labatto points given in Theorem \ref{supconv_gauss}.
Since $E_{node}$ converges with a rate $2r$, it confirms our theory in Theorem \ref{supconv_vectices}.

\begin{figure}[htbp]
\hskip-0.5cm
\scalebox{0.5}{\includegraphics{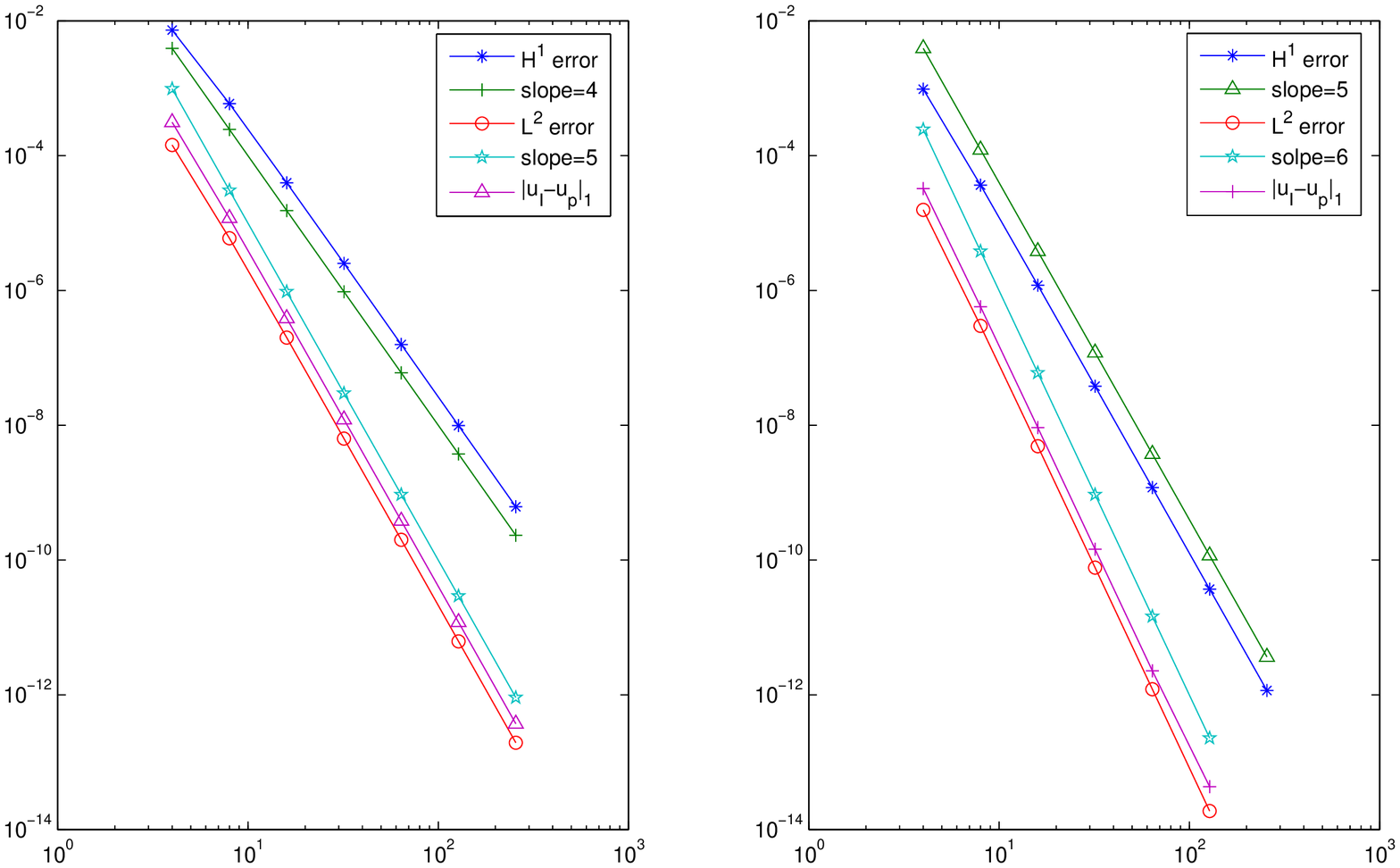}}
 \caption{left: $r=4$, right: $r=5$}\label{1_1}
\end{figure}

\begin{figure}[htbp]
\scalebox{0.5}{\includegraphics{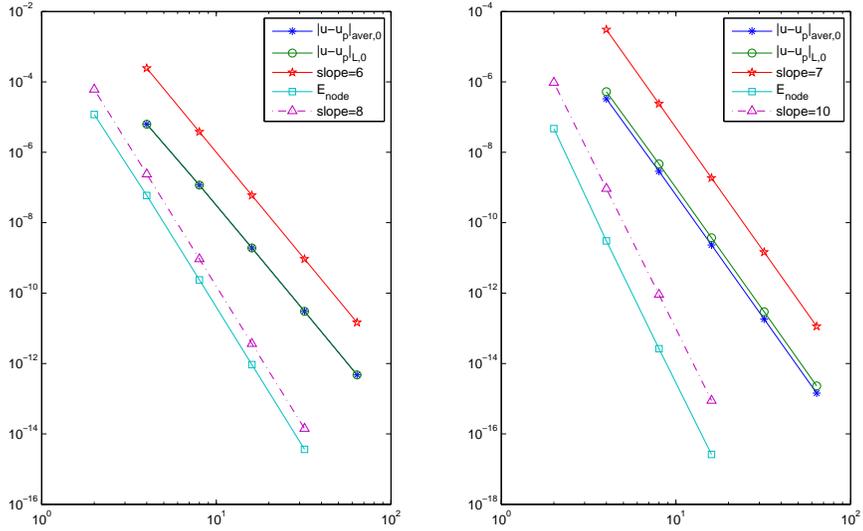}}
 \caption{left: $r=4$, right: $r=5$}\label{1_2}
\end{figure}

\bigskip

{\it Example} 2.  In this example, we test the convergence behavior of derivative error
                  at Gauss points. We consider three  cases of Equation \eqref{Poisson}, they are
\begin{itemize}
    \item[Case 1]: $\alpha(x)=e^x,\  \beta(x)=\cos x,\  \gamma(x)=x$;
    \item[Case 2]: $\alpha(x)=e^x,\  \beta(x)=0,\  \gamma(x)=x$;
    \item[Case 3]: $\alpha(x)=e^x,\  \beta(x)=0,\  \gamma(x)=0$.
\end{itemize}
  The exact solution is always $u(x)=\sin x(x^{12}-x^{11})$ and the right-hand function $f$  change according to the coefficients in different cases.

 Listed in Table \ref{case=0} are errors in the derivative approximation at Gauss points for three different cases for $r=4$ and $r=5$, respectively.
 Plotted in Fig. \ref{1_3} are corresponding error curves. We observe that the convergence rate is $r+1$ for Case 1, $r+2$ for Case 2 and $2r$ for Case 3.
 These numerical results are consistent with our theories derived in Section 4.

\begin{table}
\caption{Gauss points.} \centering
\begin{threeparttable}
        \begin{tabular}{|c|c|c|c||c|c|c|}
        \hline
        &\multicolumn{3}{|c||}{$r=4$}&\multicolumn{3}{|c|}{$r=5$}\\
        \cline{1-7}N & Case 1 & Case 2 & Case 3  & Case 1  & Case 2 & Case 3  \\
        \hline
        1 &   4.7633e-03& 4.1098e-03 &4.1493e-03&1.5667e-03&1.1105e-04 &1.0183e-04 \\
        2 &   5.4962e-04& 3.1457e-05& 2.9493e-05&2.6075e-05& 1.9562e-06& 8.6701e-09\\
        4 &   3.5877e-05& 4.2964e-07& 1.1316e-07&3.9162e-06& 2.8751e-08& 3.1732e-11\\
        8 &   1.3338e-06& 8.4296e-09& 4.3677e-10&6.7553e-08& 2.6812e-10& 3.6386e-14\\
        16 &  4.3328e-08& 1.4052e-10& 1.7002e-12&1.0779e-09& 2.1878e-12& ---\\
        32 &  1.3667e-09& 2.2319e-12& 6.6291e-15&1.6928e-11& 1.7284e-14&---\\
       \hline
       \end{tabular}
 \end{threeparttable}
 \label{case=0}
\end{table}

\begin{figure}[htbp]
\begin{center}
\hskip-0.5cm
\scalebox{0.5}{\includegraphics{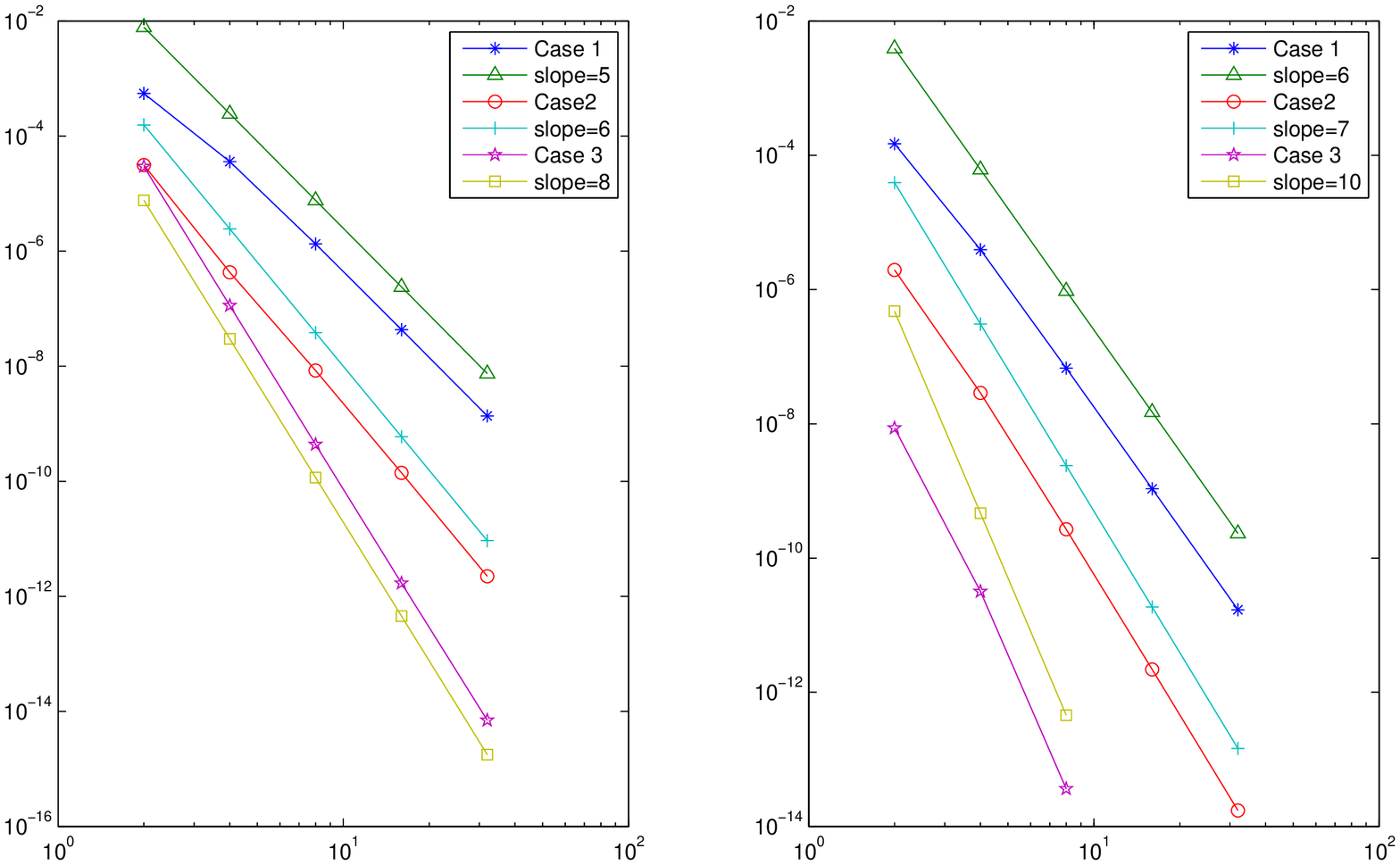}}
 \caption{left: $r=4$, right: $r=5$}\label{1_3}
\end{center}
\end{figure}

\section{Concluding remarks}
The mathematical theory for the FVM  has not been fully developed. The analysis in the literature
for high-order FVM schemes are often done case by case. It is a challenging task to develop mathematical theory for FVM scheme of an arbitrary order. In this article, we provide a
unified proof for the inf-sup condition of a family any order FVM schemes  in one dimensional setting.
Based on this, we show that the FVM solution converges to the exact solution
with optimal order, both in $H^1$ and $L^2$ norm.

We also studied the superconvergence of our FVM schemes. It is shown both theoretically and
numerically that at the nodal and interior Lobatto points, the superconvergence behavior of  FVM
is similar to that of the counterpart finite element method. Moreover, in some special cases, the superconvergence property of the derivative of the FVM solution
at the Gauss points maybe much better than that of the counterpart finite element method. For instances, when $\beta=0$, the  convergence rate of the derivative of the FVM solution is $h^{r+2}$ which is one order higher than the counterpart finite element method's $h^{r+1}$; when $\beta=\gamma=0$, the order is $h^{2r}$ which doubles the global optimal rate $h^r$, and it is much better than the counterpart finite element method's $h^{r+1}$ rate.
In a recent study \cite{ZhangZhangZou2011},  it is shown that after  a simple post-processing procedure, the  FEM solutions can have local conservation property. In this sense, the superconvergence property discovered
 in this paper become a powerful argument to support that the FVM still has its advantages.

\bigskip
\bigskip
\end{document}